\documentclass[12pt]{amsart} 
\usepackage{amsmath,amssymb,amsthm,fullpage,color}
\usepackage{graphicx}

\theoremstyle{plain}
\newtheorem{theorem}{Theorem}
\newtheorem{lemma}{Lemma}

\theoremstyle{definition}

\newcommand{\al}{\alpha}
\newcommand{\be}{\beta}
\newcommand{\ga}{\gamma}
\newcommand{\de}{\delta}

\newcommand{\eps}{\varepsilon}

\newcommand{\R}{\mathbb R}

\newcommand{\PP}{\mathcal P}

\newcommand{\brR}[1]{\!\left(#1\right)}			
\newcommand{\brS}[1]{\left[#1\right]}			
\newcommand{\brA}[1]{\left\langle#1\right\rangle}	

\newcommand{\set}[1]{\left\lbrace#1\right\rbrace}	
\newcommand{\Set}[1]{\Big\lbrace#1\Big\rbrace}		
\newcommand{\abs}[1]{\left|#1\right|}				

\newcommand{\conv}{{\rm conv}}				
\newcommand{\supp}{{\rm supp}}				
\newcommand{\spn}{{\rm span}}					

\newcommand{\dH}{{\rm d_\mathcal{H}}}			
\newcommand{\dBM}{{\rm d_\mathcal{BM}}}		

\newcommand{\hide}[1]{}

\begin{document}

\title
{\bf Stability of the reverse Blaschke-Santal\'o inequality for unconditional convex bodies}
\author{Jaegil Kim and Artem Zvavitch}\thanks{The authors are supported in part by U.S.~National Science Foundation grant DMS-1101636.}
\address{Department of Mathematical Sciences, Kent State University,
Kent, OH 44242, USA} \email{jkim@math.kent.edu}   \email{zvavitch@math.kent.edu}

\date{}

\begin{abstract}
 Mahler's conjecture asks whether the cube is a minimizer for the volume product of a body and its polar in the class of symmetric convex bodies in $\R^n$. The corresponding inequality to the conjecture is sometimes called the the reverse Blaschke-Santal\'o inequality. The conjecture is known in $\R^2$ and in several special cases. In the class of unconditional convex bodies, Saint Raymond confirmed the conjecture, and Meyer and Reisner, independently, characterized the equality case. In this paper we present a stability version of these results and also show that any symmetric convex body, which is sufficiently close to an unconditional body, satisfies the the reverse Blaschke-Santal\'o inequality.

\end{abstract}

\maketitle

\section{Introduction}

As usual,   we denote by $\brA{x,y}$  the inner product of two vectors $x, y \in \R^n$ and by $|x|$ the length of a vector $x \in \R^n$. 
A {\it convex body} is a compact convex subset of $\R^n$ with non-empty interior.  We say that a set $K$ is {\it symmetric} if it is centrally symmetric with center at the origin, i.e. $K=-K$, or, equivalently, for every $x \in K$ we get $-x \in K$. A set $K\subset \R^n$ is said to be unconditional if it is symmetric with respect to any coordinate hyperplane, i.e., $(\pm x_1, \pm x_2, \dots, \pm x_n) \in K$, for any $x \in K$ and any choice of $\pm$ signs.

We write $|A|$ for the $k$-dimensional Lebesgue measure (volume)  of a measurable set $A \subset \R^n$, where $k=1,\dots, n$  is the dimension of the minimal flat containing $A$.  The   {\it polar}  body $K^\circ$ of a symmetric convex body $K$ is defined by 
$$
K^\circ = \set{y\in\R^n \big| \brA{x,y}\le 1 \mbox{\ for all\ } x\in K}.
$$ 
The \emph{volume product} of a symmetric convex body $K$ is defined by 
$$
\PP(K)=\abs{K}\abs{K^\circ}.
$$ 
We note that the notion of the volume product (as well as polarity) can be generalized to non-symmetric setting (see for example  \cite{Sc}, p. 419), but in this paper we will focus on the questions related to the centrally symmetric bodies. It turns out that the volume product is invariant under the polarity and for any invertible linear  transformation on $\R^n$,  that is, for any $T \in {\rm GL}(n)$, 
\begin{equation}\label{eq:affine_invariance}
\PP(TK)=\PP(K)\quad\text{and}\quad\PP(K^\circ)=\PP(K).
\end{equation}
The above property makes the {\it Banach-Mazur distance} between symmetric convex bodies $K$ and $L$
  $$\dBM(K,L) = \inf \Set{d\ge1 : L\subset TK\subset dL, \mbox{ for some } T\in{\rm GL}(n)},$$
 be an extremely useful in study of properties  of the volume product. For example, the  F. John's theorem \cite{Jn} and the continuity of the volume function with respect to the  Banach-Mazur distance  guarantee that the volume product attains its maximum and minimum. The maximum for the volume product is provided by the Blaschke-Santal\'o inequality: 
\begin{equation}
\PP(K) \le \PP(B^n_2),    \mbox{ for all symmetric convex bodies } K \subset \R^n,  
\end{equation}
where $B^n_2$ is the Euclidean unit ball. The equality in the above inequality holds only for ellipsoids (see \cite{S,P,MP} for a simple proof of both the inequality and the case of equality, or a precise statement of the inequality in non-symmetric case).  Recently, the stability versions of the Blaschke-Santal\'o inequality were studied in \cite{B, BBF}.

For the minimum of the volume product, it was conjectured by Mahler in \cite{Ma1,Ma2} that $\PP(K)$ is minimized at the cube in the class of symmetric convex bodies in $\R^n$. In other words, the conjecture asks whether the following inequality is true: 
\begin{equation}\label{eq:mah}
\PP(K) \ge \PP(B_\infty^n), \mbox{ for all symmetric convex bodies } K \subset \R^n,  
\end{equation}
 where $B_\infty^n$ is the unit cube. The case of $n=2$ was proved by Mahler \cite{Ma1}. It was also proved in several special cases, like, e.g., unconditional bodies \cite{SR,Me86,R87, BF}, zonoids \cite{R86,GMR, BH}, bodies of revolution \cite{MR98} and  bodies with some positive curvature assumption \cite{St, RSW, GM}. An isomorphic version of the conjectures was proved by Bourgain and Milman \cite{BM}: there is a universal constant $c>0$ such that $\PP(K) \ge c^n\PP(B^n_2)$; see also different proofs in \cite{Ku,Na,GPV}. Functional versions of Blaschke-Santal\'o inequality and Mahler's conjecture in terms of log-concave functions were investigated by Ball \cite{Ba}, Artstein, Klartag, Milman \cite{AKM}, and Fradelizi, Meyer \cite{FM1,FM2,FM3}. For more information on Mahler's conjecture, see expository articles \cite{T1,Mak,RZ}.
 
 The local minimality of the volume product was first studied in \cite{NPRZ} by proving that the cube is a strict local minimizer of the volume product in the class of symmetric convex bodies endowed with the Banach-Mazur distance (see \cite{KR} for the local minimality at the simplex in the non-symmetric setting).  The result was used by  B\"or\"oczky and Hug \cite{BH} to provide a stability version of (\ref{eq:mah}) for zonoids, namely,  a zonoid $Z$ is close in the Banach-Mazur distance to the cube whenever  $\PP(Z)$ is close to  $\PP(B_\infty^n)$.
 
 The main goal of this paper is to provide a stability version of (\ref{eq:mah}) for unconditional convex bodies. Before stating the theorem we need to remind the definition of a Hanner polytope \cite{Ha, HL}:  a symmetric convex body $H$ is called a {\it Hanner polytope} if $H$ is one-dimensional, or it is the $\ell_1$ or $\ell_\infty$ sum of two (lower dimensional) Hanner polytopes.  It can be calculated (see for example \cite{RZ})  that the volume product of the cube is the same as that of Hanner polytopes. Thus every Hanner polytope is also a candidate for a minimizer of the volume product among symmetric convex bodies. It was also shown in \cite{Me86, R87} that Hanner polytopes are the only possible minimizers  in the class of unconditional bodies.

 In Section 2 we will prove that if the volume product of an unconditional convex body is sufficiently close  to that of the cube, then the body must be close to a Hanner polytope:
 \begin{theorem}\label{thm:stability_abs_BM}
Let $K$ be an unconditional convex body in $\R^n$. If 
$$|\PP(K)-\PP(B_\infty^n)|\le\eps$$ for small $\eps>0$,
 then there exists a  Hanner polytope  $H\subset \R^n$ such that $$\dBM(K,H) \le 1+c(n)\eps,$$
 where $c(n)>0$ is a constant depending on $n$ only.
\end{theorem}
We would like to notice that the $\R^2$-case of Theorems \ref{thm:stability_abs_BM}  was proved  as a part of more general stability result on $\R^2$ by  B\"or\"oczky, Makai,  Meyer and  Reisner in \cite{BMMR}, so in this paper we will mainly concentrate on the case $n \ge 3$.

Recently it was proved in \cite{Ki} that a Hanner polytope is a local minimizer of the volume product in the symmetric setting (see the exact statement of the theorem in the beginning of Section 3). In  Section 3 we will use  this fact and Theorem \ref{thm:stability_abs_BM}   to prove that any convex symmetric convex body, which is sufficiently close to an unconditional body, satisfies  (\ref{eq:mah}):
\begin{theorem}\label{th:extension}
Let $K$ be a symmetric convex body in $\R^n$ with 
$$\min\Set{\dBM(K,L):L\subset\R^n \text{ unconditional convex body }} = 1+ \eps,$$ for small $\eps>0$. Then,
\begin{equation*}
\PP(K)\ge(1+c(n)\eps)\cdot\PP(B_\infty^n),
\end{equation*}
where $c(n)>0$ is a constant depending on dimension $n$ only.
\end{theorem}

\section{Stability  of the reverse Blaschke-Santal\'o inequality for unconditional convex bodies}

Let  $e_1, \dots e_n$ be the standard orthonormal basis of  $\R^n$. Denote  by $\theta^\perp$ the hyperplane orthogonal to a unit vector $\theta$, and  by $K \cap \theta^\perp$,  $K | \theta^\perp$ the section of $K$ by $\theta^\perp$, the orthogonal projection of $K$ to $\theta^\perp$, respectively.  Let $\R^n_+=\{x \in \R^n: x_i \ge 0, \forall i=1, \dots,n\}$ and $K^+=K\cap \R^n_+$. To prove  Theorem \ref{thm:stability_abs_BM}  we first prove the following  lemma which is based on the inductive argument of
Meyer (see \cite{Me86}  or  \cite{RZ}):
\begin{lemma}\label{lm:stabiter}
Consider $\eps \ge 0$ and unconditional convex body  $K\subset \R^n$, $n \ge 2$ such that
$$\PP(K) \le (1+\eps)\PP(B_\infty^n).$$
Then 
\begin{equation}\label{eq:inductive_step}
\PP(K\cap e_j^\perp)\le(1+n\eps)\PP(B_\infty^{n-1}),\quad j=1,\ldots,n.
\end{equation}
\end{lemma}
\begin{proof}  
Consider $x \in K^{+}$ and $n$ pyramids created by taking the convex hull  of $x$ and the intersection of $K^+$ with each coordinate hyperplane. More precisely, let $ K^{+}_i=\conv\{x, K^{+}\cap e_i^\perp\}$ for $i=1,\ldots,n$.
  Then, using symmetry  of $K$ with respect to coordinate hyperplanes,
$$
|K|=2^n|K^+| \ge 2^n| \bigcup\limits_{i=1}^n K_i^+| =2^n \sum\limits_{i=1}^n \frac{1}{n} \brA{x,e_i}\frac{|K\cap e_i^\perp|}{2^{n-1}}= \sum\limits_{i=1}^n  \brA{x,e_i}\cdot\frac{2|K\cap e_i^\perp|}{n},$$
or equivalently,
$$
\brA{x,\sum_{i=1}^n  \frac{2| K\cap e_i^\perp|}{n |K|}\,e_i} \le 1, \mbox{ for all $x \in K^+$.}
$$
Note that, by the unconditionality of $K$, the above inequality holds for all $x\in K$, so
\begin{equation*}
\sum_{i=1}^n  \frac{2| K\cap e_i^\perp|}{n |K|}\,e_i \in K^\circ.
\end{equation*}
Applying the same argument to $K^\circ$ we get
\begin{equation*}
\sum_{i=1}^n  \frac{2| K^\circ\cap e_i^\perp|}{n |K^\circ|}\,e_i \in K.
\end{equation*}
Thus, using the  definition of polarity  we get:
\begin{equation}\label{eq:tr}
\sum\limits_{i=1}^n\frac{2| K\cap e_i^\perp|}{n |K|} \times \frac{2 |K^\circ\cap e_i^\perp|}{n |K^\circ|} \le 1.
\end{equation}
Next  notice that 
\begin{equation}\label{eq:sym}
K^\circ\cap e_i^\perp = (K | e_i^\perp)^\circ =(K \cap e_i^\perp)^\circ,
\end{equation}
where the last equality follows from the unconditionality of $K$. Finally from (\ref{eq:tr}), (\ref{eq:sym})  we get
\begin{equation}\label{eq:meyer}
\PP(K)\ge\frac{4}{n^2}\sum\limits_{i=1}^n | K\cap e_i^\perp|\times |(K \cap e_i^\perp)^\circ|=\frac4{n^2}\sum_{j=1}^n\PP(K\cap e_j^\perp).
\end{equation}
Now, we will use (\ref{eq:meyer}) to prove our lemma. Since $\PP(K)\le(1+\eps)\PP(B_\infty^n)$ and $\PP(B_\infty^n)=\frac4{n^2}\sum_{j=1}^n\PP(B_\infty^{n-1})$, we get
\begin{equation*}
\frac4{n^2}\sum_{j=1}^n(1+\eps)\PP(B_\infty^{n-1}) = (1+\eps)\PP(B_\infty^n) 
\ge \PP(K)\ge\frac4{n^2}\sum_{j=1}^n\PP(K\cap e_j^\perp).
\end{equation*}
It  implies that 
$$
\PP(K\cap e_j^\perp)\le(1+n\eps)\PP(B_\infty^{n-1}),\quad j=1,\ldots,n.
$$
Indeed, if $\PP(K\cap e_1^\perp)>(1+n\eps)\PP(B_\infty^{n-1})$, then 
\begin{equation}\label{eq:last}
\sum_{j=1}^n\PP(K\cap e_j^\perp)>(1+n\eps)\PP(B_\infty^{n-1})+(n-1)\PP(B_\infty^{n-1})= \frac{n^2}4(1+\eps)\PP(B_\infty^n),
\end{equation}
where in the first inequality we used our assumption for the $e_1^\perp$-section and the reverse Blaschke-Santal\'o inequality (\ref{eq:mah}) for unconditional convex bodies for the sections by $e_i^\perp$, $i=2, \dots, n$. Finally note that (\ref{eq:last}) together with (\ref{eq:meyer}) gives $\PP(K)>(1+\eps)\PP(B_\infty^n)$; contradiction!
\end{proof}
Next lemma will help us to treat the case of a convex body $K \subset \R^n$ whose sections by coordinate hyperplanes are close to the ($n-1$)-dimensional cube.

\begin{lemma}\label{lem:almost_cube_case}
Let $K\subset B_\infty^n$ be a convex body in $\R^n$, for $n\ge3$, satisfying 
\begin{equation}\label{eq:cond}
K\cap e_j^\perp=B_\infty^n\cap e_j^\perp,\quad\forall\, j=1,\ldots,n.
\end{equation}
Let $p=(t,\ldots,t)\in\partial K \cap \R^n_+$. Then $$|K||K^\circ|\ge \brR{1+c(1-t)}|B_1^n||B_\infty^n|,$$
where $c=c(n)>0$ is a constant depending on $n$ only.
\end{lemma}
\noindent{Remark:} The proof of Lemma \ref{lem:almost_cube_case} does not require the assumption of unconditionality of $K$. Such assumption would make the proof a bit shorter and would improve constant $c(n)$.

\begin{proof} 
Using   (\ref{eq:cond})  we claim  that $K$ contains $n$ vectors  of the form $(1,\dots,1, 0, 1,\dots, 1)$ (i.e. $n-1$ coordinates are equal to $1$ and one coordinate is  equal to $0$). Moreover, $K$ must contain the convex hull of those points. Thus $\{x: \brA{x,e_1+\cdots+e_n}=n-1\} \cap B_\infty^n \subset K$, which gives
$t \ge (n-1)/n$.  Next we choose  $q\in\partial K^\circ$  such that $\brA{p,q}=1$. Here, (\ref{eq:cond}) guarantees that the normal vector to $\partial K$ at $p$ belongs to $\R^n_+$, and thus we may assume $q\in \R^n_+$.  Let 
\begin{align*}
P&=\conv\brR{\set{p}\cup\set{x\in B_\infty^n\cap\R_+^n:\brA{x,e_1+\cdots+e_n}\le n-1}},\\
Q&=\conv\brR{\set{q}\cup B_1^n\cap\R_+^n}.
\end{align*}
Then $K\cap\R_+^n\supset P$.  Also from $K\subset B_\infty^n$ we get $K^\circ \supset B_1^n$   and thus $K^\circ\cap\R_+^n\supset Q$. Next we notice that $q$ belongs to the hyperplane with normal vector $(1/\sqrt{n},\dots, 1/\sqrt{n})$ whose distance from the origin is $1/(t\sqrt{n})$. Thus
\begin{align*}
|P| &=1-\frac1{n!}+\frac1{n!}\cdot\frac{t\sqrt{n}-\frac{n-1}{n}\sqrt{n}}{1/\sqrt{n}}= 1-\frac{1-t}{(n-1)!}, \\
|Q| &=\frac1{n!}\cdot\frac{1/(\sqrt{n}t)}{1/\sqrt{n}}=\frac1{n!t}.
\end{align*}
It gives
\begin{equation}\label{eq:symt}
|P||Q|=\frac1{n!}\brR{1+\brS{1-\frac1{(n-1)!}}\frac{1-t}{t}}\ge 4^{-n}|B_\infty^n||B_1^n|(1+c_n(1-t)),
\end{equation}
where $c_n=1-\frac1{(n-1)!}$ is a positive constant in case of $n\ge3$. 
Note also that we have  $|P||Q|\ge 4^{-n}|B_\infty^n||B_1^n|$, independently of the position of $p$ (i.e. independently of the lower/upper bounds on $t$).

If we would assume that $K$ is unconditional then we would be able to  finish the proof by simply multiplying (\ref{eq:symt}) by $4^n$. In all other cases we must split $K$ and $K^\circ$ into $4^n$ parts each depending on the choice of signs of coordinates. Construct $P$ and $Q$ corresponding to each part, compute the corresponding volumes, and take the sum.  We may apply (\ref{eq:symt}) to the part corresponding to $\R^n_+$. In all other parts we are not guaranteed the position of the point on the boundary, so we can estimate the volume using the remark after (\ref{eq:symt}). More precisely, 
$$
|K||K^\circ| =\left( \sum\limits_\delta |K_\delta|\right)\left( \sum\limits_\delta |K^\circ_\delta|\right) \ge \left( \sum\limits_\delta |P_\delta|\right)\left( \sum\limits_\delta |Q_\delta|\right) \ge \left( \sum\limits_\delta \sqrt{|P_\delta| |Q_\delta|}\right)^2,
$$
where the sum is taken over all possible choices of $n$ signs  $\delta$ and $K_\delta$ is a subset of $K$ corresponding to $\delta$. Next
\begin{align*}
|K||K^\circ| &\ge \brR{\sqrt{4^{-n}|B_\infty^n||B_1^n|(1+c_n(1-t))}+(2^n-1)\sqrt{4^{-n}|B_\infty^n||B_1^n|}}^2\\
&= |B_\infty^n||B_1^n|\brS{1+2^{-n}\brR{\sqrt{1+c_n(1-t)}-1}}^2 \\
&\ge |B_\infty^n||B_1^n|\,\brR{1+2^{-n-1}c_n(1-t)}.
\end{align*}
\end{proof}

Next we would like to review some main properties and definitions about $\ell_1$ and $\ell_\infty$ sums, Hanner polytopes and their connections to graphs. We refer to \cite{R91, Ki} for more details. 

Let $A$, $B$ be convex subsets of $\R^n$. The $\ell_1$-sum of $A$ and $B$ is defined by  $\conv(A\cup B)$, the convex hull of a set $A \cup B$,  and the $\ell_\infty$-sum is defined by  $A+B$, the Minkowski sum of $A$ and $B$.

We recall that every Hanner polytope in $\R^n$ can be obtained from $n$ symmetric intervals in $\R^n$ by taking the $\ell_1$ or $\ell_\infty$ sums. In particular, a Hanner polytope in $\R^n$ is called \emph{standard} if it is obtained from the intervals $[-e_1,e_1],\ldots,[-e_n,e_n]$ by taking the $\ell_1$ or $\ell_\infty$ sums. It is easy to see that every Hanner polytope is a linear image of a standard Hanner polytope.  Moreover, each coordinate of any vertex of a standard Hanner polytope is $0$ or $\pm1$.

The definition of a dual 0-1 polytope was given in \cite{R91} (the term `a dual 0-1 space' is used there, as a normed space with a dual 0-1 polytope): an unconditional polytope $P$ in $\R^n$ is called a \emph{dual 0-1 polytope} if each coordinate of any vertex of $P$ and $P^\circ$ is $0$ or $\pm1$. Every dual 0-1 polytope $P$ in $\R^n$ can be associated with the graph $G=G(P)$ with the vertex set $\set{1,\cdots,n}$ and the edge set defined as follows: 
\begin{center} 
$i,j\in\set{1,\ldots,n}$ are connected by an edge of $G$ if $e_i+e_j\not\in P$.
\end{center}
We note that for each $i$, $j\in\set{1,\ldots,n}$ both the section and the projection of a dual $0-1$-polytope $P$ by $\spn\set{e_i,e_j}$ are dual 0-1 polytopes, that is, $2$-dimensional $\ell_1$- or $\ell_\infty$-balls. It gives that $e_i+e_j\not\in H$ if and only if $e_i+e_j\in H^\circ$, and thus $G(P^\circ)=\overline{G(P)}$ where $\overline{G}$ denotes the compliment of $G$. 

In addition, it turned out (see Theorem 2.5, \cite{R91}) that dual 0-1 polytopes in $\R^n$ are in one-to-one correspondence with perfect graphs on $\set{1,\ldots,n}$. In particular, as a special family of dual 0-1 polytopes, the standard Hanner polytopes in $\R^n$ are in one-to-one correspondence with the graphs  on $\{1, \dots, n\}$ which do not contain any induced path of edge length $3$ (see \cite{Se} or Lemma 3.5 in \cite{R91}). Here, an induced path of a graph $G$ means a sequence of different vertices of $G$ such that each two adjacent vertices in the sequence are connected by an edge of $G$, and each two nonadjacent vertices in the sequence are not connected.

Let $G$ be the graph associated with a dual 0-1 polytope $P$ as above. It turns out (Lemma 2.5, \cite{R91}) that a point $v\in\R^n$ is a vertex of $P$ if and only if each coordinate of $v$ is $-1$, $0$, or $1$, and the support  $$\supp(v)=\set{j:\brA{v,e_j}\neq 0}$$ of $v$ is a maximal {\it independent set} of $G$, i.e., a maximal set with respect to the set-inclusion such that no pair of elements is connected by an edge.

Note that if $K$ is unconditional then we may apply a diagonal transformation to $K$ and assume that $B_1^n \subset K \subset B_\infty^n$. Indeed, this follows immediately
from the fact that a tangent hyperplane with a normal vector $e_i$  touches  $K$ at point $y_i$ which is  a  dilate of $e_i$. Thus to prove   Theorem \ref{thm:stability_abs_BM} it is enough to prove the following statement:

\medskip
\noindent{\bf Equivalent form of Theorem \ref{thm:stability_abs_BM}: }  {\it 
Let $K$ be an unconditional convex body in $\R^n$ satisfying $B_1^n\subset K\subset B_\infty^n$. If $\PP(K)\le(1+\eps)\PP(B_\infty^n)$, then there exists a standard Hanner polytope $H$ in $\R^n$ such that $\dH(K,H)=O(\eps)$.}
\medskip

\noindent Here $\dH$ is  the {\it Haussdorff distance} $\dH$ of two sets $K, L \subset \R^n$ is defined by $$\dH(K,L) = \max\brR{\max_{x\in K}\min_{y\in L}|x-y|,\,\,\max_{y\in L}\min_{x\in K}|x-y|}.$$

\begin{proof}[\bf Proof  of Theorem \ref{thm:stability_abs_BM}]  
We use induction on the dimension $n$ to prove the above statement. It is trivial if $n=1$, and the case  $n=2$ was proved in \cite{BMMR}.  Assume that the statement is true for all unconditional convex bodies of dimension $(n-1)$, $n\ge 3$.  
For the inductive step, consider an unconditional  convex body $K \subset \R^n$ satisfying
\begin{equation*}
B_1^n\subset K\subset B_\infty^n\quad\text{and}\quad\PP(K)\le(1+\eps)\PP(B_\infty^n).
\end{equation*}
We apply  Lemma  \ref{lm:stabiter} and the inductive hypothesis  to get standard Hanner polytopes $H_1\subset\R^n\cap e_1^\perp$, $\ldots$ , $H_n\subset\R^n\cap e_n^\perp$ such that 
\begin{equation}\label{eq:assumption_on_hyperdimension}
\dH(H_j, K\cap e_j^\perp)=O(\eps),\quad\text{for }j=1,\ldots,n.
\end{equation}
Now we would like to show that we can `glue' these $(n-1)$-dimensional Hanner polytopes  to create an $n$-dimensional Hanner polytope which is $O(\eps)$-close to $K$. 
For each $j$, let $G_j=G(H_j)$ be the graph associated with  $H_j$. Note that  $\set{1,\ldots,n}\setminus\set{j}$ is the vertex set of $G_j$ and that  $G_j$  does not contain an induced path of edge-length $3$. 

Consider the graph $G$ with vertex set $\set{1,\ldots,n}$ and containing all edges of $G_1,\ldots,G_n$. Our goal is to show that the polytope $P$ corresponding to $G$ is a Hanner polytope which  is $O(\eps)$- close to $K$.

We claim that the graph $G$ satisfies the following three properties.
\begin{enumerate}
\item[1.] Each $G_j$ is an induced subgraph of $G$ (i.e. for any $k$, $l\in\set{1,\ldots,n}$,  $k$ and $l$ are connected by an edge of $G_j$  if and only if they are connected by an edge of $G$).
\item[2.] If $G$ is not the path of edge-length 3, then $G$ does not contain any path of edge-length 3 as an induced subgraph.
\end{enumerate}
If $G$ is the path of edge-length 3, then it is a perfect graph. Otherwise, the second property implies that $G$ does not contain any path of edge-length 3, that is, $G$ is the graph associated with a Hanner polytope. Thus, in any cases $G$ must be a perfect graph. 

Let $P$ be the 0-1 polytope associated with $G$, i.e., $G=G(P)$. 
\begin{enumerate}
\item[3.] If $G$ is neither the complete graph nor its complement, then $\dH(K,P)=O(\eps)$. 
\end{enumerate}

Indeed, to show the first property consider  different $i,j,k\in\set{1,\ldots,n}$. Note that the existence of  $ij$ edge depends only on the properties of $K\cap\spn\set{e_i,e_j}$, so $i$, $j$ are connected by an edge of $G_k$ if and only if they are connected by an edge of $G$. It implies that each $G_k$ is an induced subgraph of $G$. 

To prove  the second property, assume that $G$ is not the path of edge-length $3$. If $G$ contains a path of edge-length $3$, then there is $j\in\set{1,\ldots,n}$ which does not belong to this path. Thus, $G_j$, as an induced subgraph of $G$, must also contain the path of edge-length $3$, which contradict with the definition of $G_j$.  

To show  the third property, we first note that if $G$ is not the complete graph then every maximal independent set of $G$ is a proper subset of $\{1,\dots, n\}$. The vertices of  $P$  have coordinates $0$, $1$ or $-1$ and the  support of a vertex  of $P$ is a maximal independent set of $G$. This implies that every vertex of $P$ has at least one zero coordinate, i.e. every vertex of $P$ is contained in one of coordinate hyperplanes. This, together with the inductive hypothesis, gives that each vertex of $P$ is $O(\eps)$-close to the boundary of $K$, which means that $P$ is contained in $O(\eps)$-neighborhood of $K$. It gives $K\supset (1+O(\eps))P$. Repeating the same argument for $\overline{G}$ we get   $K^\circ\supset (1+O(\eps))P^\circ$ and thus $\dH(K,P)=O(\eps)$.

From the above three properties, we finish the proof of  Theorem \ref{thm:stability_abs_BM}, modulo two pathological cases:
\begin{enumerate}
\item[I.] $G$ or $\overline{G}$  is the complete graph.
\item[II.] $G$ is the path of edge-length 3.
\end{enumerate}
 Therefore, it remains to consider those  cases: \\

\noindent{\bf CASE I: }{\it $G$ (respectively $\overline{G}$) is the complete graph, so all $H_1,\ldots,H_n$ are the $(n-1)$-dimensional cubes  (respectively  $(n-1)$-dimensional cross-polytopes).}   Taking the polar if necessary,  we may  assume, without loss of generality,  that all $H_1,\ldots,H_n$ are the $(n-1)$-dimensional cubes. Let $\tilde{K}=\brR{\frac1{1-\de}K}\cap B_\infty^n$, where 
$$\de=\min\set{d:(1-d)H_j\subset K\cap e_j^\perp\subset H_j,\,\,\forall j=1,\ldots,n}.
$$
Notice that $\de=O(\eps)$ and hence $\dH(K,\tilde{K})=O(\eps)$. By (linear order) continuity of the volume product in the Hausdorff metric, we get $\PP(\tilde{K})\le(1+O(\eps))\PP(K)$. Thus, the assumption $\PP(K)\le(1+\eps)\PP(B_\infty^n)$ gives
\begin{equation}\label{eq:vol_product_K_tilde}
\PP(\tilde{K})\le(1+O(\eps))\PP(B_\infty^n).
\end{equation}
Also, using $H_j=B_\infty^n \cap e_j^\perp$ and the definition of $\de$ we get
\begin{equation}\label{eq:almost_cube_in_hyperdimension}
\tilde{K}\cap e_j^\perp=B_\infty^n\cap e_j^\perp,\quad\forall\, j=1,\ldots,n.
\end{equation}
Consider  $p=(t,\ldots,t)\in\partial\tilde{K}$. Then Lemma \ref{lem:almost_cube_case} gives
$$
\PP(\tilde{K})\ge(1+c(1-t))\PP(B_\infty^n).
$$
Together with  \eqref{eq:vol_product_K_tilde} the above implies that $1-t=O(\eps)$. Thus $p$ is in $O(\eps)$-neighborhood of the vertex of $B_\infty^n$. the unconditionality of  $\tilde{K}$ and \eqref{eq:almost_cube_in_hyperdimension} give  that $\dH(\tilde{K},B_\infty^n)=O(\eps)$ and hence $\dH(K,B_\infty^n)=O(\eps)$.\\

\noindent{\bf CASE II}: {\it $G$ is the path of edge-length 3.} The direct computation below will  show that this case contradicts with assumption on the volume product of $K$. Indeed in this case  $K\subset\R^4$ is $O(\eps)$-close to the polytope $P$ which is associated with the path of edge length $3$. The volume product of $P$ is strictly greater than that of the cube $B_\infty^4$. Thus selecting $\eps_4$ small enough we will be able to contradict the assumption of the Theorem \ref{thm:stability_abs_BM}. More precisely, if $G$ is represented by the vertex set $\set{1,2,3,4}$
and the edge set $\set{12,23,34}$, then $\overline{G}$ has the same vertex set and the edge set $\set{24,41,13}$. Applying the characterization of vertices of a polytope by maximal independent sets, we can get
$$
P=\conv\Set{(\pm1,0,\pm1,0),(0,\pm1,0,\pm1)} \mbox{  and   }   P^\circ=\conv\Set{( \pm1,\pm1, 0, 0),(0,0, \pm1, \pm1)}.
$$
Then we have $|P|=|P^\circ|$ and 
\begin{align*}
|P|&=\abs{\conv\Set{(\pm1,0,\pm1,0),(0,\pm1,0,\pm1)}}\\
&=\abs{\Set{x \in \R^4: |x_1|+|x_2|\le1, |x_2|+|x_3|\le1, |x_3|+|x_4|\le1}}\\
&=2\int_{0}^1 \abs{\Set{ x \in \R^3:|x_1|+|x_2|\le1, |x_2|+|x_3|\le1, |x_3|\le1-s}} ds\\
&=4\int_{0}^1 \int_{0}^{1-s}\abs{\Set{ x \in \R^3:|x_1|+|x_2|\le1, |x_2|\le1-t}} dt ds\\
&=8\int_{0}^1 \int_{0}^{1-s} (1-t^2) dt ds=\frac{10}{3}.\\
\end{align*}
Thus $\PP(P)=\brR{\frac{10}{3}}^2 > \frac{4^4}{4!}=\PP(B_\infty^4)$. On the other hand, note that in this case $G$ and  $\overline{G}$ are not complete graphs, so we may apply the Property (3) from above to claim that  $K$ is $O(\eps)$-close to $P$. Therefore, $\PP(K)>\PP(B_\infty^4)$, which contradicts the assumption $\PP(K)\le(1+\eps)\PP(B_\infty^4)$ for $\eps>0$ small enough.\end{proof}

\section{Minimality near unconditional convex bodies}
The following local minimality result, proved in  \cite{Ki},  is the main tool of this section 
\begin{theorem}\label{th:kim}
Let $K$ be a symmetric convex body in $\R^n$ close to one of Hanner polytopes in $\R^n$ in the sense that $$\min\Set{\dBM(K,H): \text{$H$ is a Hanner polytope in $\R^n$}}=1+\eps$$ for small $\eps>0$. Then $$\PP(K)\ge\PP(B_\infty^n)+c(n)\eps$$ where $c(n)>0$ is a constant depending on the dimension $n$ only.
\end{theorem}

Let $\ga_n>0$ be a common threshold of $\eps$ to satisfy both Theorem 1 and Theorem \ref{th:kim}. More precisely, the constant $\ga_n$ satisfies the following:
\begin{itemize}
\item[] If $K$ is a symmetric convex body in $\R^n$ with $\dBM(K,H)=1+\eps$, $\eps\le\ga_n$, for some Hanner polytope $H$, then 
\begin{equation}\label{eq:Kim_stability}
\PP(K)\ge(1+\al_n\eps)\PP(B_\infty^n).
\end{equation}
\item[] If $K$ is an unconditional convex body in $\R^n$ with $\dBM(K,H)\ge1+\eps$, $\eps\le\ga_n$, for every Hanner polytope $H$, then 
\begin{equation}\label{eq:KZ_stability}
\PP(K)\ge(1+\be_n\eps)\PP(B_\infty^n).
\end{equation}
\end{itemize}
where $\al_n$, $\be_n>0$ are constants depending on $n$ only.

\begin{proof}[\bf Proof of Theorem \ref{th:extension}]
We will use $\gamma_n$ to select the constant $\eps_n$, a threshold of $\eps$ in the statement of the Theorem  \ref{th:extension}. Consider a convex symmetric body $K\subset \R^n$. Let $L \subset \R^n$ be an unconditional convex body  with smallest, among unconditional bodies,  Banach-Mazur distance to $K$ and let $\eps \ge 0$ be such that $\dBM(K,L) = 1+ \eps$.  Also consider a Hanner polytope $H_0$ with smallest, among Hanner polytopes,  Banach-Mazur distance to $L$ and let $\delta \ge 0$ be such that  $\dBM(L,H_0)=1+\de$.

We will consider two cases: $\de\le\ga_n/3$ and  $\de>\ga_n/3$.
\begin{enumerate}
\item[1.]  Assume $\de\le\frac{\ga_n}3$ and thus $\dBM(L,H_0)\le\frac{\ga_n}3$. Then
\begin{align*}
\dBM(K,H_0) &\le\dBM(K,L)\dBM(L,H_0)=(1+\eps)(1+\de)\\
&\le1+\eps+(1+\eps)\ga_n/3 \le1+\ga_n/3+2\ga_n/3=1+\ga_n,
\end{align*}
where, to guarantee  the above inequalities, we select $\eps_n  \le \min\{\frac{\gamma_n}{3},1\}$. Thus, by \eqref{eq:Kim_stability}, $$\PP(K)\ge (1+\al_n\eps)\PP(B_\infty^n).$$

\item[2.] Now assume that  $\de>\frac{\ga_n}3$ and thus $\dBM(L,H)>\frac{\ga_n}3$ for every Hanner polytope $H$. We will require
 $\eps_n\le\min\set{\frac{\be_n\ga_n}{12n},\frac1{2n}}$.
Since $L\subset TK\subset(1+\eps)L$ for some $T\in{\rm GL}(n)$, 
\begin{align*}
\PP(K)&=\PP(TK)\ge |L||(1+\eps)^{-1}L^\circ|=(1+\eps)^{-n}\PP(L)\\
&\ge (1-\eps)^n\PP(L)\ge(1-n\eps)\PP(L).
\end{align*}
Moreover, since $\PP(L)\ge (1+\be_n\ga_n/3)\PP(B_\infty^n)$ by \eqref{eq:KZ_stability},
\begin{align*}
\PP(K) &\ge (1-n\eps)\PP(L)\ge(1-n\eps)(1+\be_n\ga_n/3)\,\PP(B_\infty^n)\\
&\ge(1-n\eps)(1+4n\eps)\,\PP(B_\infty^n) =(1+3n\eps-4n^2\eps^2)\,\PP(B_\infty^n)\\
&\ge(1+n\eps)\,\PP(B_\infty^n).
\end{align*}
\end{enumerate}
Finally we combine the above two cases by taking
\begin{align*}
\eps_n&=\min\set{\frac{\ga_n}3,\frac{\be_n\ga_n}{12n},\frac1{2n}}\\
\tau_n&=\min\set{\al_n,n}.
\end{align*}
Then, $\PP(K)\ge(1+\tau_n\eps)\PP(B_\infty^n)$ whenever $\eps\le\eps_n$.
\end{proof}

\end{document}